\documentclass[12pt]{amsart}
\usepackage{amsmath}
\usepackage{amssymb}
\usepackage{amsthm}
\usepackage[shortlabels]{enumitem}
\usepackage{color}
\usepackage{tikz}
\usepackage{tikz-qtree}
\usepackage{verbatim}

\newtheorem{proposition}{Proposition}[section]
\newtheorem{corollary}[proposition]{Corollary}
\newtheorem{theorem}[proposition]{Theorem}
\newtheorem{lemma}[proposition]{Lemma}

\theoremstyle{definition}

\newtheorem{remark}[proposition]{Remark}

\newcommand{\insN}{\mathbb{N}}
\newcommand{\inN}{\in\insN}
\newcommand{\insZ}{\mathbb{Z}}
\newcommand{\insQ}{\mathbb{Q}}
\newcommand{\Spec}{\mathrm{Spec}}
\newcommand{\Max}{\mathrm{Max}}
\newcommand{\MaxPath}{\mathrm{MaxPath}}

\newcommand{\Crit}{\mathrm{Crit}}

\DeclareMathOperator{\spr}{SP-rank}
\DeclareMathOperator{\sph}{SPh}

\newcommand{\inverse}{\mathrm{inv}}
\newcommand{\tree}{\mathcal{T}}
\newcommand{\forest}{\mathcal{F}}
\newcommand{\KK}{\mathcal{K}}
\newcommand{\mm}{\mathfrak{m}}
\newcommand{\deriv}{\mathcal{D}}

\title[Ramification tree and almost Dedekind domains]{The ramification tree and almost Dedekind domains of prescribed SP-rank}
\author{Balint Rago}
\author{Dario Spirito}
\address{University of Graz, NAWI Graz, Department of Mathematics and Scientific Computing, Heinrichstraße 36,
8010 Graz, Austria}

\thanks{This work was supported by the Austrian Science Fund FWF, Project Number W1230}

\email{balint.rago@uni-graz.at}
\address{Dipartimento di Scienze Matematiche, Fisiche e Informatiche, Universit\`a di Udine, Udine, Italy}
\email{dario.spirito@uniud.it}
\date{\today}

\begin{document}
\begin{abstract}
Given a valuation $v$ with quotient field $K$ and a sequence $\KK:K_0\subseteq K_1\subseteq\cdots$ of finite extensions of $K$, we construct a weighted tree $\tree(v,\KK)$ encoding information about the ramification of $v$ in the extensions $K_i$; conversely, we show that a weighted tree $\tree$ can be expressed as $\tree(v,\KK)$ under some mild hypothesis on $v$ or on $\tree$. We use this construction to construct, for every countable successor ordinal number $\alpha$, an almost Dedekind domain $D$, integral over $V$ (the valuation domain of $v$) whose SP-rank is $\alpha$. Subsequently, we extend this result to countable limit ordinal numbers by considering integral extensions of Dedekind domains with countably many maximal ideals.
\end{abstract}

\maketitle

\section{Introduction}
Let $D$ be a Dedekind domain. Then, $D$ has \emph{prime factorization}, that is, every ideal of $D$ can be written (uniquely) as a product of prime ideals, and indeed this property characterizes Dedekind domains. A weaker property is \emph{radical factorization}: an ideal $I$ has radical factorization if it can be written as a finite product of radical ideals, and a domain $D$ is an \emph{SP-domain} if every ideal has radical factorization. Every Dedekind domain is an SP-domain, but there are integral domains that are SP-domains without being Dedekind domains; nevertheless, every SP-domain is an \emph{almost Dedekind domain}, meaning that such a domain $D$ is locally Dedekind or, equivalently, locally a discrete valuation ring.

Yet, not every almost Dedekind domain is an SP-domain. The \emph{SP-rank} of an almost Dedekind domain, introduced in \cite{SP-scattered}, is a measure of how far the domain is from having radical factorization: the SP-rank of $D$ is an ordinal number that is $1$ when $D$ is an SP-domain. The SP-rank is defined by associating to each almost Dedekind domain a chain $\{\Delta_\alpha\}_\alpha$ of subsets of the maximal space $\Max(D)$, which also defines a chain $\{T_\alpha\}_\alpha$ of overrings of $D$; the definition of the chain $\{\Delta_\alpha\}_\alpha$ is formally very similar to the definition of the derived sequence of a topological space. In this analogy, the SP-rank of an almost Dedekind domain corresponds to the Cantor-Bendixson rank of a topological space.

So far, the examples of almost Dedekind domains that are not SP-domains had a single maximal ideal that is not critical (see below for the definition), or equivalently the first set $\Delta_1$ of the chain $\{\Delta_\alpha\}$ is a singleton \cite[Examples 3.4.1 and 3.4.2]{fontana-factoring}. In particular, the SP-rank of such a domain is $2$. In this paper, we construct almost Dedekind domains whose SP-rank is any arbitrary countable ordinal number $\alpha$: in particular, we show that if $\alpha$ is a successor ordinal then the domain can be constructed as the integral closure of an (arbitrary) discrete valuation ring in an algebraic extension of its quotient field.

The main idea of the construction is to translate our problem in terms of trees. More precisely, given a discrete valuation $v$ associated to a ring $V$ and an infinite chain $\KK$ of finite extensions of the quotient field of $V$, we define a weighted tree $\tree(v,\KK)$ (which we call the \emph{ramification tree} or $v$ with respect to $\KK$) whose elements are the extensions of $v$ at the members of $\KK$, and, given two extensions $v_1,v_2$, we have $v_1\leq v_2$ if and only if $v_2$ is an extension of $v_1$; the weight of each edge is the ramification index of the extension. Using the possibility to construct extensions of valuations with prescribed ramification and inertia, we show that a tree $\tree$ can be constructed as a ramification tree under some mild hypothesis either on $\tree$ or on $v$ (Propositions \ref{prop:tree->val:balanced} and \ref{prop:tree->val:locbound}). Next, we define an SP-rank also for trees, and we link the properties of $\tree(v,\KK)$ (and, in particular, of its maximal paths) with the properties of the integral closure $A$ of $V$ in the union $K_\infty$ of the elements of $\KK$; thus, the problem of constructing an almost Dedekind domain of SP-rank $\alpha$ reduces to the construction of a tree with SP-rank $\alpha$ (Corollary \ref{cor:spr-corresp}). For successor ordinals, this is accomplished by an inductive construction (Theorems \ref{teor:sprank-balanced} and \ref{teor:sprank-countable}), while for limit ordinals it is necessary to consider a similar construction starting from a Dedekind domain with countably many maximal ideals (Theorem \ref{teor:countable-limit}). In the final Section \ref{sect:topology}, we also analyze this construction from the point of view of the (inverse) topology on the maximal space of $A$.

An alternate construction is given in our concurrent paper \cite{almded-graph}, where we construct almost Dedekind domains of arbitrary (non necessarily countable) SP-rank. However, the method in  \cite{almded-graph} only works for a specific construction, while the one in the present paper can be applied to find almost Dedekind domains in much smaller fields (see for example Corollary \ref{cor:Q})

\section{Preliminaries}
\subsection{Valuations}
Let $K$ be a field. A \emph{valuation} $v$ on $K$ is a map $v:K\longrightarrow\Gamma\cup\{\infty\}$, where $(\Gamma,+)$ is a totally ordered abelian group, such that, for every $x,y\in K$,
\begin{itemize}
\item $v(x)=\infty$ if and only if $x=0$;
\item $v(xy)=v(x)+v(y)$;
\item $v(x+y)\geq\min\{v(x),v(y)\}$.
\end{itemize}
If the map $v$ is surjective, $\Gamma$ is said to be the \emph{value group} of $v$, and is sometimes denoted by $\Gamma_v$. If $\Gamma\simeq\insZ$, then $v$ is said to be \emph{discrete} and $V$ is called a \emph{discrete valuation ring} (DVR).

The \emph{valuation ring} associated to $v$ is the ring $V:=\{x\in K\mid v(x)\geq 0\}$; the quotient field of $V$ is $K$. Conversely, if $V$ is a ring with quotient field $K$ such that, for every $x\in K$, at least one of $x$ and $x^{-1}$ is in $V$, then $V$ is called a valuation ring, and there is a valuation $v$ associated to $V$. A valuation ring is always local; we denote its maximal ideal by $\mm_V$.

Let $L$ be a field extending $K$. An \emph{extension} of $v$ to $L$ is a valuation $w$ such that $w|_K=v$; we also write $v\subseteq w$. If $W$ is the valuation ring associated to $w$, then $W\cap K=V$, and we also say that $W$ is an extension of $V$. Every extension of valuation defines a map of residue fields, $V/\mm_V\longrightarrow W/\mm_W$ and an injective map $\Gamma_v\longrightarrow\Gamma_w$. The degree $[W/\mm_W:V/\mm_V]$ is called the \emph{inertial degree} of the extension (and is denoted by $f(w/v)$) while the index $(\Gamma_w:\Gamma_v)$ is called the \emph{ramification degree}, and is denoted by $e(w/v)$. These degrees are multiplicative, i.e., if $v\subseteq v'\subseteq v''$ is a chain of extensions, then $e(v''/v)=e(v''/v')e(v'/v)$ and $f(v''/v)=f(v''/v')f(v'/v)$.

If $[L:K]<\infty$, $v$ is a discrete valuation on $K$, then $v$ has only finitely many extensions on $L$, say $w_1,\ldots,w_g$; if, moreover, $L$ is separable over $K$, their inertial and ramification degrees are linked by the \emph{fundamental equality} \cite[Chapter 6, Theorem 3]{ribenboim}
\begin{equation}\label{fundeq}
\sum_{i=1}^ge(w_i/v)f(w_i/v)=[L:K].
\end{equation}
If the extension $K\subseteq L$ is Galois, we also have $e(w_i/v)=e(w_j/v)$ and $f(w_i/v)=f(w_j/v)$ for all $i,j$ \cite[Chapter 6, J]{ribenboim}.

Conversely, we can always construct a field extension where a valuation extends with given inertia and ramification; cfr. for example \cite[Chapter 6, Theorem 4]{ribenboim}.
\begin{theorem}\label{teor:extension-valuations}
Let $v_1,\ldots,v_k$ be discrete valuations on a field $K$. Let $g_1,\ldots,g_k$, $e_{ij}$, $f_{ij}$ ($i=1,\ldots,k$, $j=1,\ldots,t_i$) be integers such that $\sum_je_{ij}f_{ij}=n$ for each $i$, for some fixed $n$. Suppose that the residue field $K_i$ of $v_i$ has extensions $L_{ij}$ such that $[L_{ij}:K_i]=f_{ij}$. Then, there is a finite separable extension $L$ of $K$ such that each valuation $v_i$ has exactly $g_i$ extensions, namely $w_{i1},\ldots,w_{ig_i}$, with ramification indices $(\Gamma_{w_{ij}}(L):\Gamma_{v_i}(K))=e_{ij}$ and such that the residue field of $w_{ij}$ is $L_{ij}$.
\end{theorem}

\subsection{Almost Dedekind domains}
Let $D$ be an integral domain. Then, $D$ is said to be an \emph{almost Dedekind domain} if $D_M$ is a discrete valuation ring for every $M\in\Max(D)$; an almost Dedekind domain is always a one-dimensional Pr\"ufer domain.

Let $D$ be an almost Dedekind domain with quotient field $K$. Any nonzero ideal $I$ of $D$ defines an ideal function
\begin{equation*}
\begin{aligned}
\nu_I\colon\Max(D)\longrightarrow & \insZ\\
M\longrightarrow & v_M(ID_M),
\end{aligned}
\end{equation*}
where $v_M:K\longrightarrow\insZ\cup\{\infty\}$ is the valuation relative to $D_M$ and $v_M(ID_M)=\inf\{v(x)\mid x\in ID_M\}$.

A maximal ideal $M$ of $D$ is said to be \emph{critical} if it does not contain any radical finitely generated ideal, or equivalently if $\sup\nu_I\geq 2$ for every finitely generated ideal $I\subseteq M$. We denote by $\Crit(D)$ the set of critical maximal ideals of $D$; this set is empty if and only if $D$ is an SP-domain (see \cite[Theorem 2.1]{olberding-factoring-SP} or \cite[Theorem 3.1.2]{fontana-factoring}).

We define recursively a chain $\{\Crit^\alpha(D)\}_\alpha$ of subsets of $\Max(D)$ and a chain $\{T_\alpha\}_\alpha$ of overrings of $D$ in the following way, where $\alpha$ is an ordinal number:
\begin{itemize}
\item $T_0:=D$, $\Crit^0(D):=\Max(D)$;
\item if $\alpha=\beta+1$ is a successor ordinal, then
\begin{equation*}
\Crit^\alpha(D):=\{P\in\Max(D)\mid PT_\beta\in\Crit(T_\beta)\};
\end{equation*}
\item if $\alpha$ is a limit ordinal, then
\begin{equation*}
\Crit^\alpha(D):=\bigcap_{\beta<\alpha}\Crit^\beta(D);
\end{equation*}
\item $\displaystyle{T_\alpha:=\bigcap\{D_M\mid M\in\Crit^\alpha(D)\}}$.
\end{itemize}
The set $\{\Crit^\alpha(D)\}_\alpha$ is a descending chain of subsets of $\Max(D)$, while $\{T_\alpha\}_\alpha$ is an ascending chain of subrings of $K$; we call the latter the \emph{SP-derived sequence} of $D$. Moreover, the maximal ideals of $T_\alpha$ are exactly the extensions of the maximal ideals in $\Crit^\alpha(D)$ \cite[Lemma 5.3]{SP-scattered}.

The \emph{SP-height} $\sph(M)$ of a prime ideal $M$ is the smallest ordinal number $\alpha$ such that $MT_\alpha=T_\alpha$, or equivalently the smallest ordinal such that $M\notin\Crit^\alpha(D)$. The \emph{SP-rank} of $D$ is the smallest ordinal number $\alpha$ such that $\Crit^\alpha(D)=\emptyset$; this rank always exists \cite[Theorem 5.1]{boundness}, and it is equal to the supremum of the SP-height of the maximal ideals of $D$. Equivalently, it is the smallest ordinal number $\alpha$ such that $T_\alpha=K$.

\subsection{Topology}
Let $D$ be an integral domain. The \emph{inverse topology} on the spectrum $\Spec(D)$ of $D$ is the topology having, as a subbasis of open sets, the sets $V(I):=\{P\in\Spec(D)\mid I\subseteq P\}$, as $I$ ranges among the finitely generated ideals of $D$. 

We denote by $\Max(D)^\inverse$ the maximal space of $D$, endowed with the inverse topology (i.e., with the subspace topology of the inverse topology of $\Spec(D)$). When $D$ is one-dimensional (in particular, when $D$ is an almost Dedekind domain) then the inverse topology on $\Max(D)$ is Hausdorff, and it coincides with the Zariski topology if and only if the Jacobson radical of $D$ is nonzero \cite[Corollary 4.4.9]{spectralspaces}; in particular, when this happens the space $\Max(D)^\inverse$ is compact.

The sets $\Crit^\alpha(D)$, where $D$ is an almost Dedekind domain, are always closed in $\Max(D)^\inverse$.

\subsection{Trees}
A \emph{tree} is a partially ordered set $(\tree,\leq)$ with a unique minimal element $r$, such that, for every $t\in\tree$, the set $\{s\in\tree\mid s<t\}$ is well-ordered; the order type of this set is called the \emph{height} of $t$. The element $r$ is called the \emph{root} of $\tree$ and is the unique element of height $0$; we also call the elements of $\tree$ its \emph{vertexes}. We denote by $\tree(\alpha)$ the set of elements with height $\alpha$. Throughout the paper, we shall assume that \emph{every vertex has finite height}, that is, that $\tree$ is what is called an \emph{$\omega$-tree}, and that $\tree$ has no maximal elements.

An \emph{edge} of $\tree$ is a pair $(a,b)$ of vertexes such that $a<b$ and there are no elements between $a$ and $b$; we denote the set of edges of $\tree$ by $E(\tree)$. A \emph{path} is a sequence $(a_i)_{i\in I}$ (where $I$ is finite or $\insN$) such that $(a_i,a_{i+1})$ is an edge for every $i\in I$; we call $a_0$ the \emph{starting point} of the path. If $\pi'\subseteq\pi$ are paths, we say that $\pi'$ is a subpath of $\pi$ and that $\pi$ is an extension of $\pi'$; if $\pi$ has no proper extensions, we say that $\pi$ is \emph{maximal}. It is easy to see that, when $\tree$ is an $\omega$-tree, every infinite path is contained in a unique maximal path, and that a path is maximal if and only if it is infinite and its starting point is the root of $\tree$. We denote by $\MaxPath(\tree)$ the set of maximal paths of $\tree$. If $\pi$ is a path and $t\in\pi$, we also say that $\pi$ \emph{passes through} $t$.

By a \emph{weight} on $\tree$ we mean a function $w:\tree\times\tree\longrightarrow\insN$ such that $w(a,b)>0$ if and only if $(a,b)$ is an edge of $\tree$; in particular, $w(a,a)=0$ for all $a$. The function $w$ can also be considered as a function from $E(\tree)$ to $\insN^+$. The \emph{(outbound) weight} of $a\in\tree$ is
\begin{equation*}
w(a):=\sum_{b\in\tree}w(a,b).
\end{equation*}
We say that $\tree$ is:
\begin{itemize}
\item \emph{locally bounded} if, for every $n\in\insN$, the set $\{w(a)\mid a\in\tree(n)\}$ is bounded;
\item \emph{balanced} if, for every $n\in\insN$, $w(a)=w(b)$ for every $a,b\in\tree(n)$;
\item \emph{totally balanced} if $w(a)=w(b)$ for every $a,b\in\tree$.
\end{itemize}

If $\pi=(a_i)_{i\in\insN}$ is a path, the \emph{weight} of $\pi$ is
\begin{equation*}
w(\pi):=\prod_{i=0}^\infty w(a_i,a_{i+1})\in\insN^+\cup\{\infty\}.
\end{equation*}
We say that $\pi$ is:
\begin{itemize}
\item \emph{finitely ramified} if $w(\pi)<\infty$;
\item \emph{unramified} if $w(\pi)=1$.
\end{itemize}
If $\pi$ is finitely ramified, then $w(a_i,a_{i+1})=1$ for all large $i$; we call the largest infinite subpath of $\pi$ that is unramified the \emph{unramified subpath} of $\pi$, and we denote it by $u(\pi)$.

If $a\in\tree$, we set $\{a\}^\uparrow:=\{b\in\tree\mid a\leq b\}$. In particular, $\{a\}^\uparrow$ is a tree with root $a$.

\section{The correspondence}\label{sect:correspond}
Let $v$ be a valuation on a field $K$, with corresponding valuation ring $V$, and consider a sequence
\begin{equation*}
\KK:K=K_0\subsetneq K_1\subsetneq K_2\subsetneq\cdots\subsetneq K_n\subsetneq\cdots
\end{equation*}
of field extensions of $K$ with $[K_{i+1},K_i]<\infty$ for every $i$. The \emph{ramification tree} of $v$ with respect to $\KK$, which we denote by $\tree(v,\KK)$ or $\tree(V,\KK)$, is defined in the following way:
\begin{itemize}
\item the elements of $\tree(v,\KK)$ are the extensions of $v$ at $K_i$, for each $i$;
\item if $v_1,v_2\in\tree(v,\KK)$, we set $v_1\leq v_2$ if $v_2$ is an extension of $v_1$;
\item if $(v_1,v_2)$ is an edge, $w(v_1,v_2)$ is the ramification index of the extension $v_1\subseteq v_2$, i.e.,
\begin{equation*}
w(v_1,v_2)=e(v_2/v_1):=[\Gamma_{v_2}:\Gamma_{v_1}],
\end{equation*}
where $\Gamma_z$ is the value group of $z$.
\end{itemize}

We can also interpret the vertexes of $\tree(v,\KK)$ in a different way: the extensions of $v$ to $K_i$ are in bijective correspondence with the maximal ideals of $A_i$, the integral closure of $V$ in $K_i$, and if $M,N$ are two such ideals, then $M\leq N$ if and only if $M\subseteq N$. 

We use the same terminology also if the chain $\KK$ is finite, i.e., if we only have a finite sequence $K=K_0\subsetneq K_1\subsetneq\cdots\subsetneq K_n$. In this case, the tree $\tree(v,\KK)$ is finite.

\begin{proposition}\label{prop:ramiftree-fund}
Preserve the notation above and let $\tree:=\tree(v,\KK)$. Then, the following hold.
\begin{enumerate}[(a)]
\item\label{prop:ramiftree-fund:edge} If $(v_1,v_2)$ is an edge and $v_1$ is a valuation on $K_i$, then $v_2$ is a valuation on $K_{i+1}$.
\item\label{prop:ramiftree-fund:locbound} $\tree$ is locally bounded.
\end{enumerate}
\end{proposition}
\begin{proof}
\ref{prop:ramiftree-fund:edge} If $(v_1,v_2)$ is an edge, then $v_2$ is a proper extension of $v_1$, and thus $v_2$ is a valuation on $K_j$ for some $j>i$. if $j>i+1$, then we have $v_1<v_2|_{K_{i+1}}<v_2$, against the fact that $(v_1,v_2)$ is an edge. Thus $j=i+1$, as claimed.

\ref{prop:ramiftree-fund:locbound} Fix $v_1\in\tree(v,\KK)$. Then,
\begin{equation*}
w(v_1)=\sum_{(v_1,v_2)\in E(\tree)}w(v_1,v_2)=\sum_{(v_1,v_2)\in E(\tree)}e(v_2/v_1),
\end{equation*}
By the previous point, if $(v_1,v_2)$ is an edge then $v_2$ is a valuation on $K_{i+1}$; by the fundamental equality \eqref{fundeq}, it follows that $w(v_1)\leq[K_{i+1}:K_i]$. Thus $\tree$ is locally bounded.
\end{proof}

Every $K_i$ is a finite extension of $K$; therefore, if $v$ is discrete, all the integral closures $A_i$ of $V$ in $K_i$ are Dedekind domains. The main object of interest of this paper is their union, or equivalently the integral closure of $V$ in the union of all $K_i$.
\begin{proposition}\label{prop:correspondence}
Preserve the notation above, let $K_\infty:=\bigcup_{i\in\insN}K_i$, and let $A$ be the integral closure of $V$ in $K_\infty$. Let $\tree:=\tree(v,\KK)$ and let $\mathcal{E}(v,K_\infty)$ be the set of extensions of $v$ to $K_\infty$. Then, there are natural bijective correspondences between $\mathcal{E}(v,K_\infty)$, $\Max(A)$ and $\MaxPath(\tree)$, given by
\begin{equation*}
\begin{aligned}
\Max(A)\longleftrightarrow & \mathcal{E}(v,K_\infty)\\
M\longrightarrow & A_M\\
\mathfrak{m}_{v_{\infty}}\cap A\longleftarrow & v_\infty
\end{aligned}
\end{equation*}
and
\begin{equation*}
\begin{aligned}
\mathcal{E}(v,K_\infty) & \longleftrightarrow\MaxPath(\tree)\\
v_\infty & \longrightarrow(v_\infty|_{K_i})_{i\in\insN}\\
\bigcup_iv_i & \longleftarrow(v_i)_{i\in\insN},
\end{aligned}
\end{equation*}
where $\left(\bigcup_iv_i\right)(x)=v_i(x)$ if $x\in K_i$.
\end{proposition}
\begin{proof}
The bijective correspondence between $\mathcal{E}(v,K_\infty)$ and $\Max(A)$ is a consequence of \cite[Theorem 20.1]{gilmer}.

If $v_\infty\in\mathcal{E}(v,K_\infty)$, then, $v_i:=v_\infty|_{K_i}$ is an extension of $v$, and $v_j$ is an extension of $v_i$ if $j>i$; therefore, $\pi:=(v_i)_{i\in\insN}$ is a maximal path on $\tree$. Conversely, if $\pi=(v_i)_{i\in\insN}$ is a maximal path on $\tree$, then the union $v_\infty:=\bigcup_iv_i$ is well-defined since $v_j$ extends $v_i$ when $j>i$. It is straightforward to see that these two correspondences are inverse one of each other.
\end{proof}

We now concentrate on the main case that is of interest for this paper, namely the case in which $v$ is a discrete valuation. In this case, we have the following.
\begin{proposition}
Preserve the notation above, and suppose that $v$ is discrete. Then, $A$ is an almost Dedekind domain if and only if every (maximal) path is finitely ramified.
\end{proposition}
\begin{proof}
If $\pi$ is the maximal path corresponding to a maximal ideal $M$ of $A$, then the ramification index of $M$ over $M\cap V$ is, by construction, equal to the ramification index of $\pi$. The result now follows from \cite[Corollary 3.6]{arnold-gilmer}.
\end{proof}

Let now $v'\in\tree(v,\KK)$. By what we have seen, we can associate to $v'$ a maximal ideal $M$ of some $A_i$; thus, we can also associate to $v'$ a finitely generated ideal of $A$, namely $MA$, and we have a natural map
\begin{equation*}
\begin{aligned}
\Psi\colon \tree(v,\KK)  & \longrightarrow \mathcal{I}_f(A),\\
v' & \longmapsto MA
\end{aligned}
\end{equation*}
where $\mathcal{I}_f(A)$ is the set of finitely generated ideals of $A$. The map $\Psi$ is in general not surjective (for example, a necessary condition for $I$ to be in the image is that $I\cap A_i$ is radical for some $i$); however, if $I$ is a finitely generated ideal of $A$, we can find an index $i$ such that $K_i$ contains all elements of a family of generators of $I$; then, $I=JA$ is the extension of $J$, an ideal of the Dedekind domain $A_i$, and thus $J$ is a product of ideals of $A_i$ associated to elements of $\tree$. In particular, $I$ can be written as a product of ideals in the form $\Psi(v_i)$, with $v_i\in\tree$.

So far, we have started from a chain of extensions and constructed a tree from it. However, under mild hypothesis we can actually go back, and construct a sequence of extensions from a tree. We give two such constructions, one with a stronger hypothesis on the tree and one with a stronger hypothesis on the valuation.
\begin{proposition}\label{prop:tree->val:balanced}
Let $\tree$ be a balanced tree, and let $v$ be a discrete valuation on $K$. Then, there is a sequence $\KK$ of finite separable extensions of $K$ such that $\tree\simeq\tree(v,\KK)$.
\end{proposition}
\begin{proof}
We construct the fields $K_i$ by induction. If $i=0$ we take $K=K_0$. Suppose that we have constructed extensions to $K_i$ such that the ramification tree of the chain $K_0\subseteq K_1\subseteq\cdots\subseteq K_i$ coincide with the first $i$ levels of $\tree$. Let $v_1,\ldots,v_k$ be the extensions of $v$ to $K_i$, corresponding respectively to $a_1,\ldots,a_k\in\tree$. For each $i$, let $(a_i,b_{i,1}),\ldots,(a_i,b_{i,t_i})$ be the edges starting from $a_i$. By the extension theorem (Theorem \ref{teor:extension-valuations}), since the tree is balanced, we can find an extension $K_{i+1}$ of $K_i$ of degree $w(a_1)=\cdots=w(a_k)$ such that each $v_i$ has $t_i$ extensions of ramification degree $w(a_i,b_{i,1}),\ldots,w(a_i,b_{i,t_i})$ and such that every extension of residue fields is trivial. The ramification tree of $K_0\subseteq K_1\subseteq\cdots\subseteq K_{i+1}$ now is isomorphic to the first $i+1$ levels of $\tree$, and the claim follows by induction.
\end{proof}

\begin{proposition}\label{prop:tree->val:locbound}
Let $\tree$ be a locallly bounded tree, and let $v$ be a discrete valuation on $K$ with finite residue field. Then, there is a sequence $\KK$ of finite separable extensions of $K$ such that $\tree\simeq\tree(v,\KK)$.
\end{proposition}
\begin{proof}
We construct the fields $K_i$ by induction. If $i=0$ we take $K=K_0$. Suppose that we have constructed extensions to $K_i$ such that the ramification tree of the chain $K_0\subseteq K_1\subseteq\cdots\subseteq K_i$ coincide with the first $i$ levels of $\tree$. Let $v_1,\ldots,v_k$ be the extensions of $v$ to $K_i$, corresponding respectively to $a_1,\ldots,a_k\in\tree$. Note that each $v_i$ is a finite extension of $v$, and thus the residue field of $v_i$ is finite: in particular, the residue field has finite extensions of any degree. Let $\delta$ be the least common multiple of $w(a_1),\ldots,w(a_k)$. For each $i$, let $(a_i,b_{i,1}),\ldots,(a_i,b_{i,t_i})$ be the edges starting from $a_i$. By Theorem \ref{teor:extension-valuations}, we can find an extension $K_{i+1}$ of $K_i$ of degree $\delta$ such that each $v_i$ has $t_i$ extensions, say $v_{i,1},\ldots,v_{i,t_i}$, such that the extension $v_i\subseteq v_{i,r}$ has ramification degree $e_r:=w(a_i,b_{i,r})$ and inertial degree $f_r:=\delta/w(a_i)$: this is possible since
\begin{equation*}
\sum_{r=1}^{t_i}e_rf_r=\sum_{r=1}^{t_i}w(a_i,b_{i,r})\frac{\delta}{w(a_i)}=w(a_i)\cdot\frac{\delta}{w(a_i)}=\delta
\end{equation*}
for every $i$.

The ramification tree of $K_0\subseteq K_1\subseteq\cdots\subseteq K_{i+1}$ now is isomorphic to the first $i+1$ levels of $\tree$, and the claim follows by induction.
\end{proof}

\section{Criticality}
In this section, we see how we can use the ramification tree to detect the SP-height of a prime ideal. We begin with a purely ring-theoretic characterization. 
\begin{proposition}\label{prop:caratt-Critalpha}
Let $A$ be an almost Dedekind domain, and let $M\in\Max(A)$. Then, $M\in\Crit^\alpha(A)$ if and only if for every finitely generated ideal $I\subseteq M$ and every $\beta<\alpha$ there is a maximal ideal $N\in\Crit^\beta(A)$ such that $\nu_I(N)\geq 2$.
\end{proposition}
\begin{proof}
Let $\{T_\alpha\}$ be the SP-derived sequence of $A$.

Suppose $M\in\Crit^\alpha(A)$, let $I\subseteq M$ be finitely generated and let $\beta<\alpha$. Then, $IT_\beta\neq T_\beta$ since $MT_\beta\neq T_\beta$. If $I$ is radical, then $MT_\beta$ is a non-critical maximal ideal of $T_\beta$, and thus $M\notin\Crit^{\beta+1}(A)\supseteq\Crit^\alpha(A)$, a contradiction. Hence $I$ is not radical, i.e., there is an $N'\in\Max(T_\beta)$ such that $\nu_{IT_\beta}(N')\geq 2$. By construction, $N:=N'\cap A\in\Crit^\beta(A)$ and $\nu_I(N)=\nu_{IT_\beta}(N')\geq 2$, and we are done.

Suppose that $M\notin\Crit^\alpha(A)$, and let $\beta$ be the SP-height of $M$. Then, $MT_\beta$ is not a critical ideal, and thus it contains a finitely generated radical ideal $J$; let $I$ be a finitely generated ideal of $A$ such that $J=IT_\beta$. If $N\in\Crit^\beta(A)$, then $\nu_I(N)=\nu_J(NT_\beta)=1$, and thus the condition of the statement does not hold.
\end{proof}

The following lemma represents the translation from ideal functions to the ramification tree.
\begin{lemma}\label{lemma:vNA}
Let $V,\KK,K_\infty,A_i,A$ as in Section \ref{sect:correspond}; suppose that $V$ is a DVR and $A$ is almost Dedekind. Let $M\in\Max(A)$ and let $N:=M\cap A_i$; let $v_i$ be the element of $\tree$ corresponding to $N$. Let $\pi$ be the path corresponding to $M$ in $\tree(v,\KK)$ and let $\pi'$ be the infinite subpath of $\pi$ starting from $v_i$. Then,
\begin{equation*}
\nu_{NA}(M)=w(\pi').
\end{equation*}
\end{lemma}
\begin{proof}
Let $N_j:=M\cap A_{i+j}$; then, $N=N_0$ and $N_0,N_1,\ldots$ is the sequence of prime ideals corresponding to $\pi'$. By definition of the ramification index, $\nu_{N_jA}(N_{j+1})=e(N_{j+1}/N_j)=w(N_j,N_{j+1})$; since $M$ is the union of the $N_j$ and the ramification indices multiply, we obtain that
\begin{equation*}
\nu_{NA}(M)=\prod_{j\geq 0}e(N_{j+1}/N_j)=\prod_{j\geq 0}w(N_j,N_{j+1})=w(\pi'),
\end{equation*}
as claimed.
\end{proof}

\begin{proposition}\label{prop:critical-tree}
Let $V,\KK,K_\infty,A_i,A$ as in Section \ref{sect:correspond}; suppose that $V$ is a DVR and $A$ is almost Dedekind. Let $M$ be a maximal ideal of $A$ and let $\pi$ be the corresponding maximal path in $\tree(v,\KK)$. Then, the following are equivalent:
\begin{enumerate}[(i)]
\item\label{prop:critical-tree:crit} $M\in\Crit(A)$;
\item\label{prop:critical-tree:upi} for all $v_i\in u(\pi)$ there is a path $\pi'$ starting from $v_i$ such that $w(\pi')>1$;
\item\label{prop:critical-tree:tail} there is an unramified subpath $\pi_0$ of $\pi$ such that, for all $v_i\in\pi_0$, there is a path $\pi'$ starting from $v_i$ with $w(\pi')>1$.
\end{enumerate} 
\end{proposition}
\begin{proof}
\ref{prop:critical-tree:crit} $\Longrightarrow$ \ref{prop:critical-tree:upi} Suppose $M\in\Crit(A)$ and let $v_i\in u(\pi)$. Then, $v_i$ corresponds to a maximal ideal $N$ of $A_i$, the integral closure of $V$ in $K_i$; let $I:=NA$. Then, $I$ is finitely generated; if $\pi_0$ is the subpath of $\pi$ starting from $N$, then $\nu_I(M)=w(\pi_0)=1$ because $\pi_0$ is a subpath of $u(\pi)$, which is unramified. Since $M$ is critical, $I$ is not radical, and thus there is a maximal ideal $M'$ of $A$ such that $\nu_I(M')\geq 2$; by construction, we must have $M'\cap A_i=N$. Let $\pi'$ be the subpath of the path corresponding to $M'$ that starts from $N$. Then, $w(\pi')=\nu_I(M')\geq 2>1$, as claimed.

\ref{prop:critical-tree:upi} $\Longrightarrow$ \ref{prop:critical-tree:tail} is obvious (just take $\pi_0=u(\pi)$).

\ref{prop:critical-tree:tail} $\Longrightarrow$ \ref{prop:critical-tree:crit} Suppose that $M$ is not critical, let $I=(x_1,\ldots,x_k)A$ be a finitely generated radical ideal contained in $M$, and fix $v_1\in\pi_0$ of level $\lambda$. Let $t$ be an integer such that $t\geq\lambda$ and such that $x_1,\ldots,x_k\in A_t$: then, there is a prime ideal $N$ of $A_t$ such that $(x_1,\ldots,x_k)A_t\subseteq N\subseteq M$. Let $v_2$ be the vertex of $\tree$ associated to $N$; then, $v_2\in\pi$ and $v_2>v_1$, so that $v_2\in\pi_0$. Applying the hypothesis on $v_2$, we can find a path $\pi'$ starting from $v_2$ such that $w(\pi')>1$. Let $M'$ be the maximal ideal of $A$ corresponding to (the maximal extension of) $\pi'$: then, $I\subseteq N\subseteq M'$, and thus since $I$ is radical we must have $\nu_I(M')=1$. However, using Lemma \ref{lemma:vNA}, $\nu_I(M')\geq\nu_{NA}(M')=w(\pi')\geq 2$, a contradiction. Therefore $M$ must be critical, as claimed.
\end{proof}

Let now $\tree$ be a tree such that all its paths are finitely ramified. We define recursively:
\begin{itemize}
\item $\Crit^0(\tree):=\MaxPath(\tree)$;
\item if $\alpha=\beta+1$, then $\Crit^\alpha(\tree):=\{\pi\in\MaxPath(\tree)\mid \forall v\in u(\pi)$ there is a path $\pi'$ starting from $v$ with $w(\pi')>1$ and such that the maximal extension of $\pi'$ is in $\Crit^\beta(\tree)\}$.
\item if $\alpha$ is a limit ordinal, then $\Crit^\alpha(\tree):=\bigcap\{\Crit^\beta(\tree)\mid \beta<\alpha\}$.
\end{itemize}
We define the \emph{SP-height} $\sph(\pi)$ of $\pi\in\MaxPath(\tree)$ as the smallest ordinal $\beta$ such that $\pi\notin\Crit^\beta(\tree)$, and the \emph{SP-rank} $\spr(\tree)$ of $\tree$ as the smallest ordinal $\beta$ such that $\Crit^\beta(\tree)=\emptyset$.

\begin{lemma}
Preserve the notation above. The SP-rank of $\tree$ is the supremum of the SP-heights of the maximal paths of $\tree$.
\end{lemma}
\begin{proof}
It follows directly from the definitions.
\end{proof}

Putting together the two propositions above we get:
\begin{theorem}
Preserve the notation above. Let $M\in\Max(A)$ and let $\pi$ be the corresponding maximal path in $\tree(v,\KK)$. Then, $M\in\Crit^\alpha(A)$ if and only if $\pi\in\Crit^\alpha(\tree(v,\KK))$.
\end{theorem}
\begin{proof}
Let $\tree:=\tree(v,\KK)$. Proposition \ref{prop:critical-tree} shows that $M\in\Crit^1(A)$ (i.e., $M$ is critical) if and only if $\pi\in\Crit^1(\tree)$.

Suppose that, for every $\beta<\alpha$, we have that $M\in\Crit^\beta(A)$ if and only if $\pi\in\Crit^\beta(\tree)$. If $\alpha$ is a limit ordinal, this immediately implies that the same holds for $\alpha$ in place of $\beta$. Suppose that $\alpha=\beta+1$ is a successor ordinal and that $M\in\Crit^\alpha(A)$. If $v'\in u(\pi)$ and $N$ is the maximal ideal associated to $v'$, then $NA$ is a finitely generated ideal contained in $M$. Then by Proposition \ref{prop:caratt-Critalpha}, there is  $M'\in\Crit^\beta(A)$ such that $\nu_{NA}(M')\ge 2$. In particular, if $\pi_0$ denotes the maximal path associated to $M'$, then $\pi_0\in\Crit^\beta(\tree(v,\KK))$ by the induction hypothesis and $\pi_0$ passes through $v'$. Let $\pi'$ denote the subpath of $\pi_0$ starting from $v'$. Then by Lemma \ref{lemma:vNA}, we have $w(\pi')=\nu_{NA}(M')\ge 2$. Hence $\pi\in\Crit^\alpha(\tree(v,\KK))$.

Conversely, suppose that $\pi\in\Crit^\alpha(\tree(v,\KK))$ and let $I=(x_1,\ldots,x_k)\subseteq M$ be a finitely generated ideal. Then there is an integer $t$ such that $x_1,\ldots,x_k\in A_t$ and such that $u(\pi)$ contains a vertex of height $t$. Let $J=(x_1,\ldots,x_k)A_t$, let $N$ be a maximal ideal of $A_t$ such that $J\subseteq N\subseteq M$ and let $v'$ be the vertex associated to $N$. Note that we have $v'\in u(\pi)$. Then $I\subseteq NA\subseteq M$ and hence there is a path $\pi'$ starting in $v'$ with $w(\pi')\ge 2$ such that the maximal extension of $\pi'$ is contained in $\Crit^\beta(\tree(v,\KK))$. Let $M'$ be the maximal ideal associated to the maximal extension of $\pi'$. Then by Lemma \ref{lemma:vNA}, we have $\nu_{I}(M')\ge\nu_{NA}(M')=w(\pi')= 2$. Since $M'\in\Crit^\beta(A)$ by the induction hypothesis, we obtain by Proposition \ref{prop:caratt-Critalpha} that $M\in\Crit^\alpha(A)$.

\end{proof}

\begin{corollary}\label{cor:spr-corresp}
Preserve the notation above. Then, $\spr(A)=\spr(\tree)$.
\end{corollary}

\section{The construction}
The results in the previous sections show that the problem of determining an almost Dedekind domain with given SP-rank $\alpha$ can be fully translated to the tree case, and that it is enough to solve it in this context. In this section, we introduce a construction that allows to build inductively trees with higher and higher SP-rank. We start with an immediate consequence of Proposition \ref{prop:critical-tree}.
\begin{corollary}\label{cor:sph-tail}
Let $\tree$ be a tree and let $\pi$ be a maximal path; let $v\in\pi$. Then, the SP-height of $\pi$ is equal to the SP-height of $\pi\cap\{v\}^\uparrow$ in $\{v\}^\uparrow$. 
\end{corollary}
\begin{proof}
Condition \ref{prop:critical-tree:tail} of Proposition \ref{prop:critical-tree} only depends on the tail of the path.
\end{proof}

Let $\tree_1,\ldots,\tree_n,\ldots$ be a sequence of trees, and let $r_i$ be the root of $\tree_i$. We construct a new weighted tree $\tree:=\Lambda(\tree_1,\ldots,\tree_n,\ldots)$ in the following way:
\begin{itemize}
\item as a set, $\tree$ is the disjoint union of $\tree_i$ (for each $i$), and of a countable sequence $(x_i)_{i=0}^n$;
\item the edges of $\tree$ are:
\begin{itemize}
\item $(a,b)$, where $a,b\in\tree_i$ and $(a,b)$ is an edge in $\tree_i$;
\item $(x_i,x_{i+1})$ for all $i\geq 0$;
\item $(x_i,r_{i+1})$ for all $i\geq 0$;
\end{itemize}
\item their weights are:
\begin{itemize}
\item if $a,b\in\tree_i$, then $w(a,b)$ is the weight of $(a,b)$ in $\tree_i$;
\item $w(x_i,x_{i+1})=1$;
\item $w(x_i,r_{i+1})=2$.
\end{itemize}
\end{itemize}

It is not hard to see that these conditions really define a tree with root $x_0$; indeed, the unique path connecting $x_0$ to $x_n$ is $(x_0,x_1,\ldots,x_n)$ (as the only edge terminating in $x_n$ is $(x_{n-1},x_n)$), while if $v\in\tree_i$ for some $i$ then we can construct a path from $x_0$ to $v$ by joining $(x_0,\ldots,x_{i-1})$ with the path from $r_i$ to $v$, and this is unique since if $(x_0,v_1,\ldots,v_n=v)$ is a path connecting $x_0$ to $v$ then there must be an edge $(v_j,v_{j+1})$ with $v_j\notin\tree_i$ and $v_{j+1}\in\tree_i$, and the unique possibility is $v_j=x_{i-1}$ and $v_{j+1}=r_i$.

An alternative way to construct $\tree=\Lambda(\tree_1,\ldots,\tree_n,\ldots)$ is by recursion:
\begin{itemize}
\item we take $x_0$ as the root;
\item $x_0$ has two direct successors, $x_1$ and $r_1$, with $w(x_0,x_1)=1$ and $w(x_0,r_1)=2$;
\item $\{r_1\}^\uparrow\simeq\tree_1$;
\item $x_1$ has two direct successors, $x_2$ and $r_2$, with weight $1$ and $2$, respectively, and $\{r_2\}^\uparrow=\tree_2$;
\item $x_2$ has two direct successors, and so on.
\end{itemize}

\begin{proposition}
Let $\pi$ be a maximal path in $\tree_i$, and let $\widetilde{\pi}$ be the maximal path in $\tree$ extending $\pi$. Then, $\widetilde{\pi}=(x_0,\ldots,x_{i-1})\cup\{\pi\}$, and $w(\widetilde{\pi})=2w(\pi)$.
\end{proposition}
\begin{proof}
As $x_0$ is the root of $\tree$, $(x_0,\ldots,x_{i-1})\cup\{\pi\}$ is a maximal path, and clearly it extends $\pi$. Moreover,
\begin{equation*}
w(\widetilde{\pi})=w(x_0,x_1)\cdots w(x_{i-2},x_{i-1})\cdot w(x_{i-1},r_i)\cdot w(\pi)=2w(\pi)
\end{equation*}
by construction. The claim is proved.
\end{proof}

\begin{proposition}\label{prop:spr-constr}
Preserve the notation above. Then,
\begin{equation*}
\spr(\Lambda(\tree_1,\ldots,\tree_n,\ldots))=\left(\limsup_{n\to\infty}\spr(\tree_n)\right)+1.
\end{equation*}
\end{proposition}
\begin{proof}
Let $\pi$ be a maximal path in $\tree:=\Lambda(\tree_1,\ldots,\tree_n,\ldots)$. If $\pi$ contains $r_n$ for some $n$, then the tail of $\pi$ is in $\tree_n$, and thus its SP-height is the same as the SP-height of $\pi\cap\{r_n\}^\uparrow$ in $\{r_n\}^\uparrow\simeq\tree_n$ by Corollary \ref{cor:sph-tail}, and in particular is at most $\spr(\tree_n)$.

Suppose that $\pi$ does not contain any $r_i$: then, $\pi=(x_i)_{i=0}^\infty$ is unramified. For every $i$, any path $\pi'\neq\pi$ starting from $x_i$ must contain the edge $(x_j,r_{j+1})$ for some $j\geq 1$; since $w(x_j,r_{j+1})=2$, $\pi'$ is not unramified. Hence, if $\beta\leq\alpha:=\limsup_n\spr(\tree_n)$ any such $\pi'$ is contained in a maximal path of SP-height $\geq\beta$. It follows that $\sph\pi>\alpha$. Moreover, there are no other paths of SP-height bigger than $\alpha$ (by the previous paragraph) and thus $\sph\pi\leq\alpha+1$. Thus $\sph\pi=\alpha+1$, and so $\spr\tree=\alpha+1$.
\end{proof}

Fix now a tree $\tree$ of SP-rank $1$. We want to construct recursively a sequence $\tree_\alpha$ of trees such that $\spr\tree_\alpha=\alpha$, where $\alpha$ is a countable successor ordinal. We thus define:
\begin{enumerate}
\item if $\alpha=1$, then $\tree_1:=\tree$;
\item if $\alpha=\beta+1$ and $\beta$ is a successor ordinal, then
\begin{equation*}
\tree_\alpha:=\Lambda(\tree_\beta,\tree_\beta,\ldots,);
\end{equation*}
\item if $\alpha=\beta+1$ and $\beta$ is a limit ordinal, then $\beta$ has countable cofinality, and thus we can find an increasing sequence $\{\gamma_n\}_{n=0}^\infty$ of ordinals whose supremum is $\beta$, and we define
\begin{equation*}
\tree_\alpha:=\Lambda(\tree_{\gamma_0+1},\tree_{\gamma_1+1},\ldots,\tree_{\gamma_n+1},\ldots).
\end{equation*}
\end{enumerate}
Note that, in the last construction, there are many possible sequences $\{\gamma_n\}$, so the sequences of trees $\{\tree_\alpha\}$ is not uniquely determined by $\tree_1$.

\begin{theorem}\label{teor:sprank-construction}
Preserve the notation above. Then, $\spr\tree_\alpha=\alpha$ for all countable successor ordinal numbers $\alpha$.
\end{theorem}
\begin{proof}
We proceed by induction on $\alpha$. If $\alpha=1$ then $\spr\tree_1=1$ by hypothesis.

If $\alpha=\beta+1$ and $\beta$ is a successor ordinal, then $\tree_\beta$ is defined and $\spr\tree_\beta=\beta$ by inductive hypothesis, so that $\spr\tree_\alpha=\spr\tree_\beta+1=\beta+1=\alpha$ by Proposition \ref{prop:spr-constr}.

Suppose $\alpha=\beta+1$ and $\beta$ is a limit ordinal. Then, $\lim_n(\gamma_n+1)=\lim_n\gamma_n=\beta$. Hence, using again Proposition \ref{prop:spr-constr} and the inductive hypothesis we have
\begin{equation*}
\spr\tree_\alpha=\left(\sup_n\spr\tree_{\gamma_n+1}\right)+1=\left(\sup_n(\gamma_n+1)\right)+1=\beta+1=\alpha.
\end{equation*}
Hence $\spr\tree_\alpha=\alpha$ for every $\alpha$, as claimed.
\end{proof}

We want to use this construction in two different ways, corresponding to Propositions \ref{prop:tree->val:balanced} and \ref{prop:tree->val:locbound}.

\begin{theorem}\label{teor:sprank-balanced}
Let $\alpha$ be a countable successor ordinal. Then:
\begin{enumerate}[(a)]
\item\label{teor:sprank-balanced:tree} there is a totally balanced tree $\tree$ with $\spr\tree=\alpha$;
\item\label{teor:sprank-balanced:ad} for every discrete valuation ring $V$, there is an almost Dedekind domain $A$ such that $V\subseteq A$ is integral and $\spr(A)=\alpha$.
\end{enumerate}
\end{theorem}
\begin{proof}
\ref{teor:sprank-balanced:tree} Let $\tree_1$ be the $\omega$-tree where each element has three direct successors and each edge has weight $1$. Then, $\spr\tree_1=1$ since all paths are unramified. Construct a sequence $\{\tree_\alpha\}$ as above; we claim that each $\tree_\alpha$ is totally balanced with weight $3$, and we proceed by induction.

The claim is trivially true for $\alpha=1$. We claim that, if $\tree':=\Lambda(\tree'_1,\ldots,\tree'_n,\ldots)$ and each $\tree'_i$ is totally balanced and has weight $3$ then the same holds for $\tree'$. Let thus $v\in\tree'$: if $v\in\tree'_i$ for some $i$ then the claim is true by induction. If $v\notin\tree'_i$, then (in the notation of the beginning of this section) $v=x_i$ for some $i$, and thus $v$ has two direct successors, $x_{i+1}$ and $r_{i+1}$, with $w(v,x_{i+1})=1$ and $w(v,r_{i+1})=2$. Thus also $w(v)=3$, and $\tree'$ is totally balanced.

Since all $\tree_\alpha$ are built with the construction $\Lambda$, by induction it follows that each $\tree_\alpha$ is balanced with weight $3$.

\ref{teor:sprank-balanced:ad} Let $\tree$ be a totally balanced tree with SP-rank $\alpha$. By Proposition \ref{prop:tree->val:balanced}, we can find a sequence $\KK$ of algebraic field extensions of the residue field $K$ of $V$ such that $\tree(V,\KK)\simeq\tree$; by the correspondence between critical sets, the integral closure $A$ of $V$ in $K_\infty:=\bigcup_{L\in\KK}L$ has SP-rank $\alpha$. The claim is proved.
\end{proof}

\begin{remark}
In the first part of the previous theorem, we constructed a totally balanced tree $\tree$ with weight $3$. To construct a sequence of trees with SP-rank $\alpha$ of weight $n>3$, it is enough to take as $\tree_1$ the tree where each element has $n$ direct successors, and slightly change the construction $\Lambda$, so that $w(x_i,r_{i+1})=n-1$ (instead of $2$) for all $i$.
\end{remark}

\begin{theorem}\label{teor:sprank-countable}
Let $\alpha$ be a countable successor ordinal. Then:
\begin{enumerate}[(a)]
\item\label{teor:sprank-countable:tree} there is a tree $\tree$ such that $\MaxPath(\tree)$ is countable and $\spr\tree=\alpha$;
\item\label{teor:sprank-countable:ad} for every discrete valuation ring $V$ with finite residue field, there is an almost Dedekind domain $A$ such that $V\subseteq A$ is integral, $\spr(A)=\alpha$ and $\Max(A)$ is countable.
\end{enumerate}
\end{theorem}
\begin{proof}
\ref{teor:sprank-countable:tree} Let $\tree_0$ be a totally ordered $\omega$-tree, and let the weight of all the edges be $1$. Construct a sequence $\{\tree_\alpha\}$ as above; we claim that $\MaxPath(\tree_\alpha)$ is countable for every $\alpha$.

Let $\tree':=\Lambda(\tree'_1,\ldots,\tree'_n,\ldots)$. We claim that if each $\MaxPath(\tree'_i)$ is countable, then also $\MaxPath(\tree')$ is countable. Indeed, if $\pi$ is a maximal path in $\tree'$, then either the tail of $\pi$ is in some $\{r_i\}^\uparrow\simeq\tree'_i$ or $\pi=(x_i)$ (where $r_i$ and $x_i$ and as in the definition of $\Lambda$); thus, $\MaxPath(\tree')$ is equal to a countable union of a countable family (the $\MaxPath(\tree'_i)$) plus another element (the path $(x_i)_{i\inN}$). Hence $\MaxPath(\tree')$ is countable. Since all the $\tree_\alpha$ are constructed with the construction $\Lambda$ and $\MaxPath(\tree_1)$ is a singleton, by induction it follows that each $\MaxPath(\tree_\alpha)$ is countable.

\ref{teor:sprank-countable:ad} Apply Proposition \ref{prop:tree->val:locbound} to the family found in the previous part of the proof.
\end{proof}

These two constructions are very general; we give an example in the next corollary.
\begin{corollary}\label{cor:Q}
Let $p$ be a prime number, and let $\overline{\insQ}$ be the algebraic closure of $\insQ$. For every countable successor ordinal number $\alpha$ there is a field $F\subseteq\overline{\insQ}$ such that the integral closure of $\insZ_{(p)}$ in $F$ is an almost Dedekind domain of SP-rank $\alpha$ with countable maximal space.
\end{corollary}
\begin{proof}
The valuation domain $\insZ_{(p)}$ has finite residue field. Hence we can apply Theorem \ref{teor:sprank-countable}.
\end{proof}

Theorems \ref{teor:sprank-balanced} and \ref{teor:sprank-countable} were given only for successor ordinals. This is unavoidable in the current setting, as we show next.
\begin{proposition}\label{prop:spr-successor}
Let $V$ be a discrete valuation ring and let $A$ be an almost Dedekind domain that is an integral extension of $V$. Then, the SP-rank of $A$ is a successor ordinal.
\end{proposition}
\begin{proof}
The maximal ideal $\mathfrak{m}$ of $V$ is contained in every maximal ideal of $A$, and thus the Jacobson radical of $A$ is nonzero. Therefore, $\Max(A)^\inverse$ is a compact space. If $\alpha=\spr(A)$ is a limit ordinal, then
\begin{equation*}
\emptyset=\bigcap_{\beta<\alpha}\Crit^\beta(A),
\end{equation*}
and no finite subintersection is empty. However, since every $\Crit^\beta(A)$ is closed in the inverse topology, this contradicts the compactness of the maximal space. Thus $\alpha$ must be a successor ordinal.
\end{proof}

\section{Limit ordinals}
As shown in Proposition \ref{prop:spr-successor}, the previous construction cannot give us almost Dedekind domains with a limit ordinal as its SP-rank. To do so, we must start with a Dedekind domain instead of a DVR. We consider this problem in a slightly more general way. See \cite[Section 6.3]{fontana-factoring} and \cite{starloc} for the definition and for properties of Jaffard families.

\begin{proposition}\label{prop:spr-jaffard}
Let $A$ be an almost Dedekind domain and let $\Theta$ be a Jaffard family on $A$. Then:
\begin{enumerate}[(a)]
\item for every $P\in\Max(A)$, we have $\sph(P)=\sph(PT)$, where $T$ is the only element of $\Theta$ such that $PT\neq T$;
\item $\spr(A)=\sup\{\spr(T)\mid T\in\Theta\}$.
\end{enumerate}
\end{proposition}
\begin{proof}
Clearly the second statement is a direct consequence of the first one.

We first show that $P$ is critical if and only if $PT$ is critical. If $P$ is not critical,  there is a finitely generated radical ideal $I\subseteq P$; then, $IT$ is a finitely generated radical ideal of $T$ contained in $PT$, and thus $PT$ is not critical. Conversely, if $PT$ is not critical, then there is a finitely generated radical ideal $J\subseteq PT$. Then, $J':=J\cap A$ is radical, and since $T$ is a Jaffard overring of $A$, $J'$ is finitely generated too \cite[Lemma 5.9]{starloc}; since $J'\subseteq PT\cap A=P$ we have that $P$ is not critical.

Suppose now that $\sph(Q)=\sph(QT)$ whenever $Q\neq QT$ and either $\sph(Q)<\alpha$ or $\sph(QT)<\alpha$. Suppose that $\sph(P)\geq\alpha$ and $PT\neq T$. Let $\{A_\alpha\}$ and $\{T_\alpha\}$ be the SP-derived sequences of $A$ and $T$, respectively; the inductive hypothesis imply that $T_\alpha=TA_\alpha$. In particular, $\sph(P)=\alpha$ if and only if $P$ is not critical in $A_\alpha$, if and only if $P$ is not critical in $TA_\alpha=T_\alpha$, if and only if $\sph(PT)=\alpha$. By induction, $\sph(P)=\sph(PT)$ for all $P$. 
\end{proof}

\begin{corollary}\label{cor:ic-ded}
Let $D$ be a Dedekind domain with quotient field $K$, $\KK=\{K_n\}_{n\in\insN}$ be a sequence of finite field extensions of $K$, and let $A$ be the integral closure of $D$ in $K_\infty:=\bigcup_nK_n$. Then,
\begin{equation*}
\spr(A)=\sup_{M\in\Max(D)}\spr(\tree(D_M,\KK))
\end{equation*}
\end{corollary}
\begin{proof}
Let $\Theta:=\{A_M\mid M\in\Max(D)\}$, where $A_M:=(D\setminus M)^{-1}A$; then, $A_M$ is the integral closure of $D_M$ in $K_\infty$. Then, $\Theta$ is a Jaffard family on $A$, and thus by Proposition \ref{prop:spr-jaffard} we have $\spr(A)=\sup\{\spr(A_M)\mid M\in\Max(D)\}$. By Corollary \ref{cor:spr-corresp}, $\spr(A_M)=\spr(\tree(D_M,\KK))$, and the claim follows.
\end{proof}

Unlike in the DVR case, when $D$ is not semilocal we cannot fully control the extension of all valuations induced by localizations of $D$, as there are infinitely many of them. However, we can modify our construction to consider only finitely many valuations at a time.
\begin{theorem}\label{teor:countable-limit}
Let $D$ be a Dedekind domain with countably many maximal ideals and let $\alpha$ be a nonzero countable ordinal. Then, there is an almost Dedekind domain $A$ such that $D\subseteq A$ is integral and $\spr(D)=\alpha$.
\end{theorem}
\begin{proof}
We proceed by induction on $\alpha$. If $\alpha=1$ it is enough to take $A=D$.

Let now $\tree_1,\ldots,\tree_n,\ldots$ be a sequence of totally balanced trees with the same weight $d$, and let $\Max(D)=\{M_1,\ldots,M_n,\ldots\}$; let also $v_i$ be the valuation relative to $M_i$. Using the extension theorem, we construct a sequence $\KK=\{K_n\}_{n\in\insN}$ of extensions in the following way:
\begin{itemize}
\item $K_1$ is an extension where $v_1$ is ramified as the first level of $\tree_1$;
\item $K_2$ is an extension of $K_1$ where:
\begin{itemize}
    \item each extension of $v_1$ ramifies as the second level of $\tree_1$:
    \item all extensions of $v_2$ to $K_1$ extend as the first level of $\tree_2$;
\end{itemize}
\item in general, $K_{i+1}$ is an extension of $K_i$ where:
\begin{itemize}
    \item $v_1$ ramifies as the $(i+1)$-th level of $\tree_1$;
    \item each extension of $v_2$ to $K_1$ extends as the $i$-th level of $\tree_2$;
    \item each extension of $v_3$ to $K_2$ extends as the $(i-1)$-th level of $\tree_2$;
    \item each extension of $v_{i+1}$ to $K_i$ extends as the first level of $\tree_{i+1}$.
\end{itemize}
\end{itemize}
Note that at each step we are always controlling finitely many valuations, since $[K_i:K]<\infty$ and thus $v_j$ has only finitely many extensions to $K_i$ for all $j,i$.

Fix now a $j$. Then, $v_j$ has finitely many extensions to $K_{j-1}$, say $v_{j,1},\ldots,v_{j,t}$, and by construction $\tree(v_{j,i},\{K_{j+k}\}_{k\in\insN})\simeq\tree_j$. In particular, every maximal path $\pi$ of $\tree(v_j,\KK)$ ends up in a tree isomorphic to $\tree_j$; it follows that $\spr(\tree(v_j,\KK))=\spr(\tree_j)$. Since $\spr(\tree(v_j,\KK))=\spr(\tree(D_{M_j},\KK))$, by Corollary \ref{cor:ic-ded} we have
\begin{equation*}
\spr(A)=\sup_{j\in\insN}\spr(\tree(D_{M_j},\KK))=\sup_{j\in\insN}\spr(\tree_j).
\end{equation*}

If $\alpha$ is a successor ordinal, by Theorem \ref{teor:sprank-balanced} we can find a tree $\tree$ with SP-rank $\alpha$, and take $\tree_n:=\tree$ for every $n$; if $\alpha$ is a limit ordinal, take a sequence $\gamma_n$ of successor ordinals with limit $\alpha$ (notice that $\alpha$ has countable cofinality), and use Theorem \ref{teor:sprank-balanced} to find trees with $\spr\tree_n=\gamma_n$. In both cases, the claim follows by the previous equality.
\end{proof}

\section{Topology}\label{sect:topology}
Let $\tree$ be an $\omega$-tree, and consider the set $\MaxPath(\tree)$ of maximal paths of $\tree$. For each $a\in\tree$, let
\begin{equation*}
V(a):=\{\pi\in\MaxPath(\tree)\mid a\in\pi\}.
\end{equation*}

\begin{lemma}
The family $\{V(a)\mid a\in\tree\}$ is closed under taking finite intersections, and thus is a basis of open sets for a topology on $\MaxPath(\tree)$.
\end{lemma}
\begin{proof}
Let $a,b\in\tree$. If $V(a)\cap V(b)$ is nonempty, there is a maximal path $\pi$ containing both $a$ and $b$; since a path is totally ordered, we have $a\leq b$ or $b\leq a$. In the former case $V(a)\supseteq V(b)$, while in the latter $V(a)\subseteq V(b)$; in both cases, $V(a)\cap V(b)$ belongs to the family.
\end{proof}

We call this topology the \emph{order topology} on $\MaxPath(\tree)$.

\begin{lemma}\label{lemma:clopen}
Suppose that each element of $\tree$ has only finitely many direct successors. Then, each $V(a)$ is closed in the order topology.
\end{lemma}
\begin{proof}
Suppose that $a$ is of level $i$: then, the hypothesis guarantees that there are only finitely many elements of level $i$, say $a,b_1,\ldots,b_k$. Then, every maximal path must pass either through $a$ or through some $b_i$; therefore, $V(a)$ is the complement of $V(b_1)\cup\cdots\cup V(b_k)$, which is open. Hence $V(a)$ is also closed.
\end{proof}

\begin{proposition}\label{prop:omef}
Let $V$ be a discrete valuation ring on $K$ and let $\KK=\{K_n\}_{n\in\insN}$ be a chain of finite extensions of $K$. Let $A$ be the integral closure of $V$ in $K_\infty:=\bigcup_nK_n$. Then, the natural correspondence $\Phi:\Max(A)\longrightarrow\MaxPath(\tree(V,\KK))$ of Proposition \ref{prop:correspondence} is a homeomorphism when $\Max(A)$ is endowed with the inverse topology and $\MaxPath(\tree(V,\KK))$ with the order topology.
\end{proposition}
\begin{proof}
Let $a\in\tree$. Then, $a$ corresponds to a prime ideal $P$ of $A_n$, the integral closure of $V$ in $K_n$ (for some $n$), and a path $\pi$ contains $a$ if and only if the corresponding maximal ideal $M:=\Phi^{-1}(\pi)$ contains $P$. Hence, $\Phi^{-1}(V(a))=V(PA)$. Moreover, $P$ is finitely generated (since it is an ideal of a Dedekind domain) and thus $V(PA)$ is open in $\Max(A)^\inverse$. Hence $\Phi$ is continuous.

To prove that $\Phi$ is open, let $I=(x_1,\ldots,x_k)$ be a finitely generated ideal of $A$, let $n$ be an integer such that $x_1,\ldots,x_k\in K_n$, and let $A_n$ be the integral closure of $V$ in $K_n$; then, $J:=(x_1,\ldots,x_k)A_n$ is an ideal of $A_n$, and thus $J=P_1^{e_1}\cdots P_t^{e_t}$ for some prime ideals $P_1,\ldots,P_t$ of $A_n$; hence,
\begin{equation*}
\Phi(V(I))=\Phi(V(JA))=\Phi\left(\bigcup_{i=1}^tV(P_iA)\right)=\bigcup_{i=1}^t\Phi(V(P_iA)).
\end{equation*}
Let $b_i$ be the element of $\tree$ corresponding to $P_i$; by the previous part of the proof, $\Phi(V(P_iA))=V(b_i)$; hence, $\Phi(V(I))$ is a union of open sets, and thus it is itself open. Thus, $\Phi$ is a homeomorphism.
\end{proof}

Let $X$ be a topological space. A point $x\in X$ is \emph{isolated} if $\{x\}$ is an open set; if $x$ is not isolated, $x$ is called a \emph{limit point}. The set of limit points of $X$ is called the \emph{derived set} of $X$, and is denoted by $\deriv(X)$. More generally, if $\alpha$ is an ordinal number, we set
\begin{itemize}
\item $\deriv^0(X):=X$;
\item if $\alpha=\beta+1$ is a successor ordinal, $\deriv^\alpha(X):=\deriv(\deriv^\beta(X))$;
\item if $\alpha$ is a limit ordinal, $\displaystyle{\deriv^\alpha(X):=\bigcap_{\beta<\alpha}\deriv^\beta(X)}$.
\end{itemize}
The smallest ordinal $\alpha$ such that $\deriv^\alpha(X)=\deriv^{\alpha+1}(X)$ is called the \emph{Cantor-Bendixson rank} of $X$. If $X=\deriv(X)$ (i.e., if $X$ has no limit points) then $X$ is said to be \emph{perfect}.
\begin{lemma}
Let $\tree$ be a tree and $\pi\in\MaxPath(\tree)$. Then, $\pi$ is isolated if and only if there is a $v\in\tree$ such that $\pi$ is the only maximal path containing $v$.
\end{lemma}
\begin{proof}
If there is such $v$, then $V(v)=\{\pi\}$ and $\pi$ is isolated. If $\pi$ is isolated, then $\{\pi\}$ is an open set. Since the family $\{V(a)\}$ is a basis that is closed by finite intersection, it must be $\{\pi\}=V(v)$ for some $v$; hence $v\in\pi$ and $\pi$ is the only maximal path through $v$.
\end{proof}

\begin{proposition}
Let $\alpha$ be a countable successor ordinal, and let $V$ be a discrete valuation domain with finite residue field. Then, there is an almost Dedekind domain $A$ with SP-rank $\alpha$ such that $\Crit^\beta(A)=\deriv^\beta(\Max(A)^\inverse)$ for every $\beta$.
\end{proposition}
\begin{proof}
Let $\tree_\alpha$ be the sequence constructed in the proof of Theorem \ref{teor:sprank-countable}. We proceed by induction on $\alpha$.

If $\alpha=1$ then $\MaxPath(\tree)=\{\tree\}$ and the claim is obvious.

Suppose that $\alpha=\gamma+1$, so that $\tree:=\tree_\alpha:=\Lambda(\tree_{\gamma_1},\ldots,\tree_{\gamma_n},\ldots)$ for a sequence $\{\gamma_n\}$ with limit $\gamma$ (where, if $\gamma$ is a successor ordinal, we have $\gamma_i=\gamma$ for all $i$). Set $X:=\MaxPath(\tree)$. For each $i$, consider the set $X_i:=\{\pi\in X\mid \pi\cap\tree_{\gamma_i}\neq\emptyset\}$. Then, $X_i=V(r_i)$ is closed by Lemma \ref{lemma:clopen}. Moreover, $X_i$ is naturally homeomorphic to $\MaxPath(\tree_{\gamma_i})$. Hence, the Cantor-Bendixson rank of $\pi\in X_i$ (as an element of $X$) is the same as its Cantor-Bendixson rank as its restriction to $X_i$; since the SP-rank too only depends on the tail of the path, it follows that for any $\pi\in X_i$ we have $\pi\in\deriv^\beta(X)$ if and only if $\pi\in\Crit^\beta(\tree)$.

Suppose now that $\pi\notin X_i$ for all $i$; therefore, $\pi=(x_i)_{i=1}^\infty$ and its SP-height is $\alpha$. We claim that $\pi\in\deriv^\alpha(X)$. Let $\alpha_0=\beta_0+1$ be the Cantor-Bendixson height of $\pi$. Then, $\pi$ is an isolated point of $\deriv^{\beta_0}(X)$, and thus there is a basic open set $V(x)$ such that $V(x)\cap\deriv^{\beta_0}(X)=\{\pi\}$; thus, $x=x_i$ for some $i$. By construction, $V(x_i)$ contains all paths containing $x_i$, and thus $X_i\subseteq V(x_i)$. Suppose $\beta_0<\beta$, i.e., $\alpha_0<\alpha$. If $\gamma$ (the ordinal just before $\alpha$) is a successor, then $\gamma_i=\gamma$ and thus $X_i$ contains elements of Cantor-Bendixson height $\gamma$, a contradiction. If $\gamma$ is a limit ordinal, then there is a $j$ such that $\gamma_j>\beta_0$; since $V(x_j)\subseteq V(x_i)$, we have $V(x_j)\cap\deriv^{\beta_0}(X)=\{\pi\}$, while $V(x_j)\cap\deriv^{\beta_0}(X)$ should contain some element of $X_j$, again a contradiction. Therefore, we must have $\alpha_0=\alpha$, and $\pi\in\deriv^\alpha(X)$. By induction, the claim is proved.
\end{proof}

The situation for the construction carried on in Theorem \ref{teor:sprank-balanced} is, on the other hand, completely different.
\begin{lemma}\label{lemma:maxpath-compact}
Let $\tree$ be an $\omega$-tree such that every element has only finitely many direct successors. Then, $\MaxPath(\tree)$ is compact.
\end{lemma}
\begin{proof}
We can consider $\tree$ as a weighted tree by giving to each edge a weight of $1$; since there are only finitely many elements of level $n$, for every $n$, $\tree$ is locally bounded. Let $v$ be a discrete valuation with finite residue field; by Proposition \ref{prop:tree->val:locbound}, we can construct a sequence $\KK$ of extensions such that $\tree(v,\KK)\simeq\tree$. By Proposition \ref{prop:omef}, $\MaxPath(\tree)\simeq\Max(A)^\inverse$, where $A$ is the integral closure of $V$ in $\bigcup_nK_n$. However, since the Jacobson radical of $A$ is nonzero, the inverse topology on $\Max(A)$ coincides with the Zariski topology, and in particular $\Max(A)^\inverse$ is compact. Thus also $\MaxPath(\tree)$ is compact.
\end{proof}

We denote by $\mathbf{2}^\omega$ the space of all countable $\{0,1\}$-sequences, i.e., the product of countably many copies of $\{0,1\}$; the space $\mathbf{2}^\omega$ is also homeomorphic to the Cantor set inside $\mathbb{R}$.
\begin{proposition}
Let $\tree$ be an $\omega$-tree such that every element has at least two but only finitely many direct successors. Then, $\MaxPath(\tree)\simeq\mathbf{2}^\omega$.
\end{proposition}
\begin{proof}
We show that $\MaxPath(\tree)$ is non-empty, perfect, compact, totally disconnected, and metrizable; these properties characterize $\mathbf{2}^\omega$ \cite[Chapter 30]{willard}.

Clearly $\MaxPath(\tree)$ is non-empty, and it is compact by Lemma \ref{lemma:maxpath-compact}. To show that it is perfect, we must show that it has no isolated points; however, if $\pi$ is an isolated point, there should be an $a$ such that $V(a)=\{\pi\}$, i.e., such that there is a unique maximal path containing $a$; this contradicts the fact that $a$ has at least two direct successors.


By Lemma \ref{lemma:clopen}, $\MaxPath(\tree)$ has a basis of clopen subsets; since it is also $T_1$ (for each $\pi$ we have $\{\pi\}=\bigcap\{V(a)\mid a\in\pi\}$), by \cite[Theorem 29.5]{willard} $\MaxPath(\tree)$ is totally disconnected. 

By \cite[Lemma 29.6]{willard}, $\MaxPath(\tree)$ is also Hausdorff, and being compact it is also normal (hence regular). Furthermore, it is second countable since $\tree$ is countable and thus the set of all the $V(a)$ is countable. By Urysohn's metrization theorem (see e.g. \cite[Theorem 23.1]{willard}) $\MaxPath(\tree)$ is also metrizable. 

Therefore, $\MaxPath(\tree)$ is non-empty, perfect, compact, totally disconnected, and metrizable, and thus homeomorphic to $\mathbf{2}^\omega$.
\end{proof}

\begin{corollary}
Let $\tree$ be one of the trees $\tree_\alpha$ constructed in the proof of Theorem \ref{teor:sprank-balanced}. Then, $\MaxPath(\tree)\simeq \mathbf{2}^\omega$.
\end{corollary}
\begin{proof}
By construction, every element of $\tree$ has at least two direct successors. Hence, we can apply the previous proposition to all the trees $\tree_\alpha$, and thus $\MaxPath(\tree)\simeq\mathbf{2}^\omega$.
\end{proof}

In particular, if $A$ is one of the almost Dedekind domains constructed in Theorem \ref{teor:sprank-balanced}, then $\Max(A)^\inverse\simeq\mathbf{2}^\omega$.

\bibliographystyle{plain}
\bibliography{biblio}
\end{document}